\documentclass{amsart}
\usepackage{amssymb}

\newtheorem{theorem}{Theorem}[section]
\newtheorem{lemma}[theorem]{Lemma}
\newtheorem{proposition}[theorem]{Proposition}
\newtheorem{corollary}[theorem]{Corollary}
\newtheoremstyle{notauto}{}{}{\itshape}{}{\bfseries}{.}{0.5em}{\thmnote{#3}}
\theoremstyle{notauto}

\theoremstyle{definition}
\newtheorem{definition}[theorem]{Definition}

\theoremstyle{remark}
\newtheorem{remark}[theorem]{Remark}

\begin{document}
\title{Commuting graph of $A$-orbits}
\author{\.{I}sma\.{I}l \c{S}. G\"{u}lo\u{g}lu}
\address{\.{I}sma\.{I}l \c{S}. G\"{u}lo\u{g}lu, Department of Mathematics, Do%
\u{g}u\c{s} University, Istanbul, Turkey}
\email{iguloglu@dogus.edu.tr}
\author{G\"{u}l\.{I}n Ercan$^{*}$}
\address{G\"{u}l\.{I}n Ercan, Department of Mathematics, Middle East
echnical University, Ankara, Turkey}
\email{ercan@metu.edu.tr}
\thanks{$^{*}$Corresponding author}
\maketitle

\begin{abstract}
Let $A$ be a finite group acting by automorphisms on the finite group $G$.
We introduce the commuting graph $\Gamma (G,A)$ of this action and study
some questions related to the structure of $G$ under certain graph
theoretical conditions on $\Gamma (G,A)$.
\end{abstract}

\section{introduction}

Throughout the article all groups are finite. There have been a lot of research to
investigate the effect of the commutativity relation on the structure of a
group. In \cite{Ab} the noncommuting graph of a group $G$ was introduced as
the simple graph with the vertex set $G\setminus Z(G)$ where two distinct vertices 
$x$ and $y$ forming an edge $\{x,y\}$ if and only if they do not commute.
This is the complementary graph of the so called commuting graph of $G.$ The
graph theoretical invariants of such graphs and also the characterization of
groups with a given commuting or noncommuting graph have been studied
extensively by various authors (e.g \cite{Ab}, \cite{Dar}, \cite{SW}, \cite%
{CP}, \cite{PaMo}).

In the present paper we introduce a generalization of these graphs, namely
we define \textit{the commuting graph of $A$-orbits on $G$} as follows:

\begin{definition}
Let $A$ be a group acting by automorphisms on the group $G$. The commuting
graph $\Gamma(G,A)$ of this action is the graph with vertex set $%
V(\Gamma(G,A))=\{x^{A}: x\in G\setminus \{1\}\}$, the set of all $A$-orbits on $%
G\setminus \{1\}$, where two distinct vertices $\mathcal{O}$ and $\mathcal{%
O^{\prime}}$ are joined by an edge (written $\mathcal{O}\sim\mathcal{%
O^{\prime}}$) if and only if there exist $x\in\mathcal{O}$ and $y\in\mathcal{%
O^{\prime}}$ such that $xy=yx$.
\end{definition}

Clearly we have $\Gamma(G,A)=\Gamma(G,A/C_{A}(G)).$ In case $A$ is trivial, $%
\Gamma(G,A)$ is the commuting graph of $G$ as we define it, while the
standard commuting graph of $G$ is the induced subgraph of $\Gamma(G,\{1\})$%
\ on the subset $G\backslash Z(G)$ of the vertex set $G\setminus \{1\}.$ If $A$
is equal to the group of inner automorphisms of $G,$ then $\Gamma(G,A)$ is
exactly the commuting graph of nontrivial conjugacy classes of $G$
introduced by Herzog et al. in \cite{HeLoMa}.

Let $(V,E)$ be a simple graph. For any partition $\overline{V}=\{V_{1},V_{2},%
\ldots,V_{r},\ldots\}$ of $V$ one can define a new
graph $(\overline{V},\overline{E})$, which we call the quotient graph modulo 
$\overline{V}$, where $\{V_{i},V_{j}\}$ is an edge in $\overline{E}$ if and
only if there exist $x\in V_{i}$ and $y\in V_{j}$ such that $\{x,y\}$ is an
edge in $E$. This method of constructing a new graph from a given graph can
be used to explain the relation between several graphs associated to a group. For
example if one consider the commuting graph of a group $G$ as the simple
graph $\Gamma (G)=(V,E)$ with vertex set $V=G\backslash\left\{ 1\right\} $
where two distinct vertices $g$ and $h$ form and edge if and only if $gh=hg$ then $%
\Gamma(G,A)$ is by definition the quotient graph modulo the partition $%
\left\{ x^{A}:x\in G\backslash\left\{ 1\right\} \right\} $ of $V.$ More
generally if $A\leq B\leq AutG$ then $\Gamma(G,B)$ is the quotient of $%
\Gamma(G,A)$ modulo the partition of $V(\Gamma(G,B))$ of $V(\Gamma(G,A))$ as
the $B$-orbit $x^{B}$ is a union of some $A$-orbits for any $x\in
G\backslash\left\{ 1\right\}. $

There is another graph associated to a given group $G$, namely the Gr\"
unberg-Kegel graph (the prime graph) of $G$ denoted by $GK(G)$ the vertex
set of which is the set $\pi(G)$ of prime divisors of the order of $G$ where
two distinct vertices $p$ and $q$ are joined by an edge if $G$ contains an
element of order $pq.$ Since there exists an element in $G$ of order $pq$ if
and only if there are elements $x$ and $y$ in $G$ of orders $p$ and $q$
respectively, such that $xy=yx$, the Gr\" unberg-Kegel graph carries a lot
of information related to the relation of commutativity and hence is closely
related to $\Gamma(G,A).$ Let $V_{0}$ denote the set of $A$-orbits of
elements of prime order in the group $G$ and let $%
\Gamma=(V_{0},E_{0}) $ be the subgraph induced from $\Gamma(G,A)$ on $V_{0}$%
. It should be noted that the quotient graph $(\overline{V}_{0},\overline{E}%
_{0})$ modulo the partition $\overline{V}_{0}=\{V_{p}:p\in\pi(G)\},$ where $%
V_{p}$ is the set of $A$-orbits in $G$ of elements of order $p,$ is
isomorphic to $GK(G).$

This first work on the commuting graph associated to a group action is
essentially devoted to the study of some questions related to the structure
of $G$ under certain graph theoretical conditions on $\Gamma (G,A)$. In
Section $2$ we study the connectedness of $\Gamma (G,A)$ when $G$ is
solvable. In Section $3$ we investigate the case where $\Gamma (G,A)$
contains a complete vertex, that is, a vertex $z^{A}$ which is adjacent to
every other vertex. As a dual concept, in Section $4$, we consider the case
where $\Gamma (G,A)$ contains an isolated vertex. And finally in the last
section we handle triangle free commuting graphs of $A$-orbits.

\section{connectedness}

In this section we study the connectedness of $\Gamma(G,A)$ if $G$ is a
solvable group.

\begin{lemma}
(i) Let $B\leq A$. If $\Gamma(G,B)$ is connected of diameter $d$ then $%
\Gamma(G,A)$ is connected of diameter at most $d$.\newline
(ii) If $\Gamma(G,B)$ has $m$ connected components then $\Gamma(G,A)$ has at
most $m$ connected components.
\end{lemma}

\begin{proof}
This follows from the fact that $x^{B}\sim y^{B}$ implies $x^{A}\sim y^{A}$.
\end{proof}

\begin{theorem}
Suppose that $G$ is a solvable group. Then\\
$(i)\,  \Gamma (G,A)$ is disconnected if and only if $G$ is Frobenius or $2$%
-Frobenius. In any case, the number of connected components of $\Gamma (G,A)$
is $m+1$ where $m$ is the number of $A$-orbits on the set of Frobenius
complements of lower Frobenius subgroup.\\
$(ii)$ If $\Gamma (G,A)$ is connected then it is of diameter at most $8$.
\end{theorem}

\begin{proof}
$(i)$ Suppose that $G$ is either Frobenius or $2$-Frobenius. In the former
case $G=KL$ where $K\lhd G$ and $L$ acts semiregularly on $K.$ In the latter
case $G=KLM$ where $K\lhd G$ and $KL\lhd G$ such that the groups $KL$ and $%
LM $ are both Frobenius. In any case $C_{G}(x)\leq L$ for every $x\in L$.
Let $y\in K$. Suppose that there is a path joining $x^{A}$ and $y^{A}$, that
is there are $x=x_{1},\ldots ,x_{m}=y$ in $G$ such that $x_{1}^{A}\sim
x_{2}^{A}\sim \cdots \sim x_{m}^{A}=y^{A}$. Notice that $x_{1}^{A}\sim
x_{2}^{A}$ implies $[x_{1},{x_{2}}^{a}]=1$ for some $a\in A$ and hence ${%
x_{2}}^{a}\in L.$ It follows that $x_{2}$ lies in an $A$-conjugate of $L$.
Similarly one can see that for every $i=1,\ldots ,m$, $x_{i}$ belongs to an $%
A$-conjugate of $L$, which contradicts the fact that $y=x_{m}\in K.$
Therefore $\Gamma (G,A)$ is disconnected in case where $G$ is Frobenius or $%
2 $-Frobenius.

Conversely assume that $\Gamma(G,A)$ is disconnected. Then so is $%
\Gamma(G,\{1\}) $ by Lemma 2.1. It follows that $Z(G)=1$ and so $\Gamma(G,\{1\})$
coincides with the commuting graph of $G.$ By \cite{CP} it is known that
such a group $G$ is Frobenius or $2$-Frobenius.

Suppose now that $\Gamma (G,A)$ is disconnected. If $G=KL$ is Frobenius with
kernel $K$ and complement $L$, then the vertices $x^{A}$ and $y^{A}$ lie in the same connected
component if $x$ and $y$ are $A$-conjugate to elements of $L$ as $Z(L)\ne 1.$ Let $A$ have $%
m $ orbits, represented by $
L=L_{1},L_{2},\ldots ,L_{m}$, on the set of Frobenius complements. Then the above observation shows that for any two
elements $x$ and $y$ in $G\setminus K$ the vertices $x^{A}$ and $y^{A}$ are
connected to each other if and only if they intersect the same $L_{i}.$ Thus
the number of connected components in this case is $m+1.$ If $G=KLM$ is $2$-Frobenius
then the elements of the $A$-invariant subgroup $KL$ are distributed in
exactly $m+1$ connected components if $m$ is the number of $A$-orbits on the
set of Frobenius complements of $KL.$ Let now $x$ be an element of $%
G\setminus KL.$ Then $xK $ is conjugate in $G/K$ to an element of $MK/K.$ So
we can assume without loss of generality that $x\in KM\setminus K$. Clearly $%
x=yz$ for some $y\in K$ and $z\in M.$ As $L\left\langle z\right\rangle $ is
a Frobenius group acting on $V=\Omega _{1}(Z(K))$ with $L$ acting fixed
point freely on $V$ we see that $z$ and hence $x$ centralizes an element of $%
Z(K)$. Thus $x^{A}$ lies in the unique component containing the $A$-orbits
lying in $K.$ So we have again $m+1$ connected components as claimed.

$(ii)$ follows from \cite{CP}.
\end{proof}

\section{when $\Gamma(G,A)$ contains a complete vertex}

In this section our main goal is to characterize the groups $G$ for which $%
\Gamma(G,A)$ is complete for some $A\leq AutG.$  We first investigate a
special case, namely the existence of a complete vertex. For $z\in
G\setminus \{1\}$ the vertex $z^{A}\in V(\Gamma(G,A))$ is said to be a
complete vertex if and only if $z^{A}$ is adjacent to every vertex, and this
holds if and only if $G=\bigcup_{a\in A}{C_{G}(z)}^{a}$. In particular if $%
A/C_{A}(G)\leq Inn(G),$ then $z^{A}$ is complete if and only if $G=C_{G}(z)$%
, that is $z\in Z(G)$ since the union of conjugates of a proper subgoup of a
finite group cannot cover the whole group. This is not true for arbitrary $%
A\leq Aut(G)$ in general. For example, let $G$ be the quaternion group and $%
A $ the subgroup of automorphisms of $G$ of order $3.$ Then $A$ has three
orbits on $G\setminus \{1\},$ each of which is a complete vertex and $%
\Gamma(G,A)$ is complete although $G$ is not abelian.

As a preparation we investigate the influence of the structure of $C_G(z)$
on the structure of $G$ if $z$ is an element of $G$ such that $z^A$ is
almost complete, that is, $z^A$ is adjacent to every vertex $x^A$ for
elements $x$ of prime power order.

Throughout the paper we use the following lemma without any further reference.

\begin{lemma}
Let $H$ and $N$ be $A$-invariant subgroups of $G$ where $N\unlhd G.$ Suppose
that $x^{A}\sim y^{A}$ in $\Gamma(G,A)$. Then\\
$(i)\,  (xN)^{A}\sim (yN)^{A}$ in $\Gamma(G/N,A)$ if $x,y\in G\setminus N$;\\
$(ii)\,  x^{A}\sim y^{A}$ in $\Gamma(H,A)$ if $x,y\in H$.
\end{lemma}
\newpage
\noindent \textbf{A condition for nilpotency}\newline

\begin{theorem}
Suppose that for any two distinct primes $p$ and $q$ and for all $x,y\in
G\setminus \{1\}$ where $x$ is a $p$-element and $y$ is a $q$-element, we
have $x^A\sim y^A$. Then $G$ is nilpotent.
\end{theorem}

\begin{proof}
We use induction on the order of $GA$ and proceed over a series of
steps:\newline

\textit{(1) Every proper $A$-invariant subgroup of $G$ is nilpotent. For any $A$%
-invariant normal subgroup $N$ of $G,$ the group $G/N$ is nilpotent.}

\begin{proof}
They follow easily by induction.
\end{proof}

\textit{(2) $G$ has a unique minimal normal $A$-invariant subgroup, say $M$,
where $M$ is an elementary abelian $p$-group for some prime $p$ and $G/M$ is
nilpotent.}

\begin{proof}
Let $M_1$ and $M_2$ be two distinct minimal normal $A$-invariant subgroups
of $G$. Both $G/M_1$ and $G/M_2$ are nilpotent by (1) and hence so is $G$%
, which is a contradiction. Therefore there is a unique minimal normal $A$%
-invariant subgroup, say $M$, of $G$. It follows that $M$ is an elementary
abelian $p$-group for some prime $p$ or $M=G$. Suppose that the latter
holds. Then $G$ is a characteristically simple group, that is, $%
G=G_1\times\cdots \times G_k$ where each $G_i$ is a nonabelian simple group
isomorphic to $G_1$, and $A$ acts transitively on the set $\{G_1,\ldots
,G_k\}$.

Set $A_1=Stab_A(G_1).$ We observe now that the group $G_1$ satisfies the
hypothesis of the theorem with respect to the action of $A_1$: Pick $x,y\in
G_1$ where $x$ is a $p$-element and $y$ is a $q$-element respectively. Let $%
\{t_{1}=1,\ldots ,t_{k}\}$ be a right transversal for $A_{1}=Stab_{A}(G_{1})$
in $A$. Set $G_{i}={G_{1}}^{t_{i}}$ for $i=1,\ldots ,k.$ Then $%
X=\prod_{i=1}^{k}{x}^{t_{i}}$ and $Y=\prod_{i=1}^{k}{y}^{t_{i}}$ are $p$-
and $q$-elements of $G$, respectively and hence there exists $a\in A$ such
that $[X,Y^{a}]=1$ by hypothesis. Clearly we have $A=\bigcup_{i=1}^{k}{{t_{i}%
}^{-1}}A_{1}$. Suppose that $a\in{{t_{i_{0}}}^{-1} }A_{1}.$ Then $a={{%
t_{i_{0}}}^{-1}}b$ for some $b\in A_{1}$ and 
\begin{equation*}
Y^{a}=\prod_{i=1}^{k}({y}^{t_{i}})^{a}=(\prod_{i_{0}\neq i=1}^{k}{y}^{{t_{i}{%
t_{i_{0}}}^{-1}b}})\cdot y^{b}.
\end{equation*}
Notice that $\prod_{i_{0}\neq i=1}^{k}{y}^{{t_{i}{t_{i_{0}}}^{-1}b}}\in
\prod_{i=2}^{k}G_{i}$. Since $[X,Y^{a}]=1$, we have $[x,y^{b}]=1$. This
establishes the claim that $G_{1}$ satisfies the hypothesis with respect to
the action of $A_1.$ It follows then that $G_1=G$ by induction, that is, $G$
is a nonabelian simple group.

Notice that by hypothesis, the Gr\" unberg-Kegel graph $GK(G)$ of $G$ is
complete. One can observe that groups $G$ of prime order are the only simple
groups such that $GK(G)$ is complete: Indeed, if $G$ is a nonabelian simple group and $2$ is a complete vertex of $GK(G)$, then it follows from Theorem 7.1 in \cite%
{VaV} that $G=A_n$ for some $n$ such that there is no prime in $[n-3,n]$. On the other hand, Corollary 7.6 (2) in \cite{VaV} shows that there is no such simple group. 
\end{proof}
\newpage
\textit{(3) $G=MQ$ where $Q$ is an elementary abelian $q$-group for a prime $%
q\ne p$ with $[M,Q]=M$ and $C_Q(M)=1$.}

\begin{proof}
The group $G/M=\bar{G}$ is nilpotent by (1). Then $\bar{G}=\bar{G}_p\times \bar{G}_{p^{\prime}}$ where $%
\bar{G}_p=P/M$ and $\bar{G}_{p^{\prime}}=QM/M$ for some $Q\in
Hall_{p^{\prime}}(G).$ If $MQ\ne G$ then it is nilpotent and hence $[M,Q]=1.$
On the other hand $[P,Q]\leq M$ which yields
that $[P,Q]=[P,Q,Q]=1$ and hence $G$ is nilpotent. Thus we have $G=MQ.$ Clearly
by induction we see that $Q$ is a $q$-group for some prime $%
q\ne p$, $[M,Q]=M$, $C_Q(M)=1$ and $Q=\Omega_1(Z(Q))$, as claimed.
\end{proof}

\textit{(4) Final contradiction.}

\begin{proof}
Regarding $M$ as an irreducible $GA$-module we write $M_{_{Q}}=%
\bigoplus_{i=1}^mW_i$ as a direct sum of its $Q$-homogeneous components.
Clearly $Q/C_Q(W_i)$ is cyclic of order $q$ for each $i.$ Let $%
x=\sum_{i=1}^{m}x_i\in M$ with $1\ne x_i\in W_i$ for each $i.$ Now $%
C_G(x)=MC_Q(x).$ Let $1\ne z\in C_Q(x).$ Then for each $i$ we have $z\in
C_Q(x_i)=C_Q(W_i)$ and hence $z\in C_Q(M)=1,$ which is impossible. Therefore 
$C_Q(x)=1$, that is, $M=C_G(x)$. This contradicts the hypothesis that $%
x^A\sim y^A$ for any $y\in Q$ and completes the proof.
\end{proof}
\end{proof}

As a direct consequence of the above theorem we have the following.

\begin{theorem}
If $\Gamma(G,A)$ is complete for some $A\leq AutG$ then $G$ is nilpotent.
\end{theorem}

Another upshot of Theorem 3.2 which is also of independent interest can be given as follows.

\begin{corollary}
Suppose that there exists a vertex $z^A$ which is adjacent to every vertex $%
x^A$ for elements $x$ of prime power order. If $C_G(z)$ is nilpotent then so
is $G$.
\end{corollary}

\begin{proof}
Let $x$ and $y$ be $p$- and $q$-elements of $G$ for distinct primes $p$ and $%
q.$ By hypothesis, $x^A\sim z^A\sim y^A,$ that is there are $a,b\in A$
such that $x^a$ and $y^b$ are both contained in $C_G(z).$ It follows by the
nilpotency of $C_G(z)$ that $x^a$ and $y^b$ commute. Then Theorem 3.2
implies that $G$ is nilpotent as desired.
\end{proof}

\begin{remark}
Let $A=G$ where $G$ is a nonabelian nilpotent group. Then $\Gamma(G,A)$
coincides with the commuting graph on the conjugacy classes which is defined
in \cite{HeLoMa}. Since $A=G$, the class $z^A$ is a complete vertex if and
only if $z\in Z(G).$ Therefore $\Gamma(G,A)$ is not complete and hence the
converse of Theorem 3.3 is not true. More precisely there exist nonabelian
nilpotent groups $G$ (for example $G=D_8$) such that $\Gamma(G,A)$ is not
complete for any $A\leq AutG$ ($\Gamma (D_{8},Aut(D_{8}))$ has $3$ vertices
corresponding to the orbits of the involution in the center of $D_{8},$ the
elements of order $4$, and the involutions outside the center of $D_{8},$
the last two of the vertices are not adjacent).

On the other hand there exists a pair $(G,A)$ where $G$ is nonabelian for
which $\Gamma(G,A)$ is complete: Let $G$ be an extraspecial group of order $%
3^{3}$ and of exponent $3.$ Then $AutG$ contains a subgroup $A$ which is
isomorphic to $Q_{8}$ and acts transitively on the set of nontrivial
elements of $G/Z(G).$ So $\Gamma(G,A)$ is complete. This example also shows
that in Corollary 3.4 $G$ need not be abelian if one assumes that $C_G(z)$
is abelian.
\end{remark}
\newpage
\noindent \textbf{A condition for solvability}\\

We shall now present an analogue of Theorem 3.1.

\begin{theorem}
Suppose that for any two distinct primes $p$ and $q$ and for all $x,y\in
G\setminus \{1\}$ where $x$ is a $p$-element and $y$ is a $q$-element, there
exists $a\in A$ such that the group $\langle x, y^{a}\rangle$ is solvable.
Then $G$ is solvable.
\end{theorem}

\begin{proof}
Let $G$ be a minimal counterexample to the theorem, and let $N$ be a minimal
normal $A$-invariant subgroup of $G.$ Let $xN$ and $yN$ be nontrivial $p$-
and $q$-elements of $G/N$ respectively, for distinct primes $p$ and $q$. Replacing $x$ and $y$ by suitable powers, we may assume that $x$ is a $p$%
-element and $y$ is a $q$-element. Then by hypothesis there is $a\in A$ such
that the group $\langle x, y^{a}\rangle$ is solvable. This forces that the
group $\langle xN, (yN)^{a}\rangle$ is also solvable. Therefore the group $G/N$
satisfies the hypothesis of the theorem and hence is solvable. It follows then that $N$ is nonsolvable and hence $N=N_1\times \cdots \times N_m$ such that $%
N_i\cong N_1$ for each $i$ where $N_1$ is a nonabelian simple group with $%
m\geq 1.$ By Theorem B in \cite{DoGuHePra} we get distinct primes $p$ and $q$
dividing $|N_1|$ such that $\langle u, v\rangle$ is nonsolvable for all $%
u,v\in N_1 $ of orders $p$ and $q$ respectively. Set now $x=(x_1,\ldots
,x_m) $ and $y=(y_1,\ldots ,y_m)$ in $N$ where $x_i$ and $y_i$ are elements
of $N_i $ of orders $p$ and $q$, respectively, for $i=1,\ldots ,m.$ Now $%
\langle x, y^{a}\rangle$ is a solvable subgroup of $N$ for some $a\in A$ and
its projection to $N_i$ is $\langle x_i, {y_j}^{a}\rangle$ for some suitable 
$j$, and hence is nonsolvable. This forces that $\langle x, y^{a}\rangle$ is
nonsolvable, which is a contradiction.
\end{proof}

An immediate consequence of Theorem 3.6 is the following analogue of
Corollary 3.4.

\begin{corollary}
Suppose that there exists a vertex $z^A$ of $\Gamma(G,A)$ which is adjacent to every vertex $%
x^A$ for elements $x$ of prime power order. If $C_G(z)$ is solvable then so
is $G$.
\end{corollary}

\begin{proof}
Let $x$ and $y$ be $p$- and $q$-elements of $G$ for distinct primes $p$ and $%
q.$ By hypothesis, there exist $a$ and $b$ in $A$ such that $x^a$ and $y^b$
are both contained in $C_G(z).$ Then the group $\langle x,
y^{ba^{-1}}\rangle $ is solvable. It follows by Theorem 3.6 that $G$ is
solvable.
\end{proof}

\begin{remark}
Observe that in Corollary 3.4 and Corollary 3.7 the essential property of
the group $H=C_G(z)$ is that the set $\bigcup_{a\in A}H^a$ contains all the
elements of prime power order in $G$, and not it is the centralizer of an
element. So instead of these corollaries one could have proven the following interesting result.
\end{remark}

\begin{proposition}
Let a group $A$ act on the group $G$, and let $H$ be a subgroup of $G$ such
that $\bigcup_{a\in A}H^a$ contains all the elements of prime power order in 
$G$. Then $G$ is solvable (resp. nilpotent) if $H$ is solvable (resp. nilpotent).
\end{proposition}

\noindent \textbf{A consequence of the completeness of a vertex}\\

The following lemma will be needed in the proof of Proposition 3.11 which is
obtained under the assumption that $\Gamma(G,A)$ contains a complete vertex $%
z^A$ where $z$ is a $p$-element for some prime $p.$

\begin{lemma}
Let $G$ be a nonabelian simple group and let $\alpha\in Aut(G)$ of coprime
order. Then $\pi(C_{G}(\alpha))$ is a proper subset of $\pi(G).$
\end{lemma}

\begin{proof}
If $G$ is a simple group having an automorphism $\alpha$ of prime order $p$
coprime to $\left\vert G\right\vert $ then $G$ is a group of Lie type over a
field with $q^{p}$ elements and $\alpha$ is an automorphism arising from the
field automorphism of order $p$ where $C_{G}(\alpha)$ is a group of the same
Lie type over the field with $q$ elements. In Table 6 of \cite{Atl}, the
orders of Chevalley groups are listed and are products of cyclotomic
polynomials evaluated at certain powers of the cardinality of the defining
field. Looking at the primitive prime divisors of these polynomials one can
easily check that there exists a prime dividing $|G|$ which does not divide $%
|C_G(\alpha)|$. This establishes the claim.
\end{proof}

\begin{proposition}
Let $z^{A}$ be a complete vertex of $\Gamma(G,A)$ where $z$ is a $p$-element
for a prime $p$ such that $C_G(z)$ contains a Sylow $p$-subgroup of $G$.
Then $z\in O_{p}(G).$

\begin{proof}
We use induction on $|G|+|z|$, and proceed over a series of steps:\newline

\textit{(1) $O_{p}(G)=1$ and $|z|=p.$}

\begin{proof}
An induction argument applied to $G/O_p(G)$ shows that $O_{p}(G)=1$. If $%
|z|\ne p$, $(z^p)^A$ is also a complete vertex of $\Gamma(G,A)$ and so by induction we get $z^p\in O_{p}(G)=1.$
\end{proof}

\textit{(2) Let $M$ be a minimal normal $A$-invariant subgroup of $G$ and $P$ be a Sylow $p$-subgroup of $G$ such that $z\in Z(P).$ Then $G=MP$ and $%
G/M=\langle (zM)^{A}\rangle$ is an elementary abelian $p$-group.}

\begin{proof}
If $z\in M$ then by induction applied to $\Gamma(M,A)$ we get $M=G$. If $%
z\notin M$ then by induction applied to $\Gamma(G/M,A)$ we get $zM\in
O_{p}(G/M)=Y/M.$ Then $Y=M(P\cap Y)$ and $z\in Z(P\cap Y).$ Let $T/M=\langle (zM)^{A}\rangle$. Clearly we have $T/M\leq Z(Y/M).$ Notice that $T$ cannot be a proper
subgroup of $G$ because otherwise induction applied to $\Gamma(T,A)$ gives
that $z\in O_{p}(T)\leq O_{p}(Y)\leq O_p(G),$ which is not possible. So $G/M=\langle (zM)^{A}\rangle$ is elementary abelian.
\end{proof}

\textit{(3) $M$ is nonsolvable.}

\begin{proof}
Suppose first that $M$ is an elementary abelian $q$-group for some prime $q.$
Clearly $q\neq p$ and $z\notin M$. Let $%
M_{_{P}}=M_{1}\oplus\cdots\oplus M_{s}$ be the direct sum decomposition of $M$ into its $P$-homogeneous components.
Now $C_{M_{i}}(z)=1$ or $M_{i}$ as $M_{i}$ is a sum of isomorphic
irreducible modules.

Pick $1\neq x_{i}\in M_{i}$ for each $i=1,\ldots ,s.$ Set $x=\Sigma
_{i=1}^{s}x_{i}.$ Then $x^a\in C_{M}(z)=\bigoplus_{i=1}^s{M_i}(z).$ Since $A$ acts on the set $\{M_1,\ldots ,M_s\}$, we get $C_{M_i}(z)\ne 1$ and hence $C_{M_i}(z)=M_i$ for all $i.$ It follows that $z\in C_P(M)=O_p(G)=1.$ This proves that $M$ is nonsolvable.
\end{proof}

\textit{(4) $M$ is nonabelian simple if $p$ divides $|M|.$}

\begin{proof}
Suppose that $M=N_{1}\times\cdots\times N_{s}$ with isomorphic nonabelian
simple groups $N_{i},$ $i=1,\ldots ,s$. $P\cap M$ is a Sylow $p$-subgroup of $M$ and hence $P\cap N_{i}$
is a Sylow $p$-subgroup of $N_{i},$ $i=1,\ldots ,s.$ If $P\cap M\neq 1$ then
for each $i$ we have $P\cap N_{i}\neq 1$, and as $z\in Z(P)$ we
see that $z$, and hence each $z^{a},a\in A$, normalizes each $N_{i}$, that is $%
s=1.$ Therefore $M$ is nonabelian simple if $p$ divides $|M|,$ as
claimed.
\end{proof}

\textit{(5) $P\cap M\neq 1.$}

\begin{proof}
Suppose the contrary. Since $z^{A}$ is a complete vertex we have $%
M=\bigcup\nolimits_{a\in A}C_{M}(z)^{a}.$ Then for any nonidentity element $x\in N_{1}$
there exists $a\in A$ such that $x^{a}\in C_{M}(z)$. But then $x^{a}\in
N_{1}^{a}=N_{k}$ for some $k\in \{1,\ldots ,s\}$, and hence $N_{k}$ is left invariant by $z$ and $%
x^{a}\in C_{N_{k}}(z)$. By Lemma 3.9 there exists a prime $q$ dividing the
order of $N_{k}$ (which is equal to the order of $N_{1}$) which does not
divide the order of $C_{N_{k}}(z)$ since $z$ induces a coprime automorphism
of the simple group $N_{k}$ as $P\cap M=1$. Therefore if we choose an
element $x$ of $N_{1}$ of order $q$ it cannot lie in $\bigcup_{a\in A}C_{M}(z)^{a}$. 
\end{proof}

\textit{(6) Final contradiction.}

\begin{proof}
If $z\in M$ then we see that the nonabelian simple group $G$ contains $p$ as
a complete vertex in its prime graph $GK(G)$. As it follows from \cite{VaV}, 
$G=A_{n}$ for some $n$, and $p=2$. Let $\sigma$ be a $k$-cycle in $G$ where $%
k=n$ if $n$ is odd, and $k=n-1$ if $n$ is even. In any case we see that $%
C_{G}(\sigma)=\langle\sigma\rangle$. But $\sigma$ has to lie in $%
\bigcup_{a\in A}C_{G}(z)^{a}$ and hence there exists $a\in A$ such that 
$z^{a}\in C_{G}(\sigma)$ which is not possible.

So we are left with the case $P\cap M\neq 1$ but $z\notin M$ and $M$ is a
nonabelian simple group. If $p$ is odd, it follows by \cite{GlauGur} that the automorphism
of $M$ induced by $z$ is inner, and so there is $x\in M$ such that $\tau_{x^{-1}}=z$. This gives that $zx\in
C_{G}(M)=1$ which is impossible. Thus we have $p=2$. Since $M=\bigcup\nolimits_{a\in
A}H^{a} $ where $H=C_{M}(z),$ we see that $\pi (M)=\pi (H).$ Now one can
invoke Corollary 5, Table 10.7 in \cite{Liebeck} to see that the only possibility is $M\cong
PSU(4,2)$ and $H\cong S_{6}.$ But this cannot happen since $M$ contains an
element of order $12$ and $S_{6}$ does not, and the proof is complete.
\end{proof}
\end{proof}
\end{proposition}

\section{when $\Gamma (G,A)$ contains an isolated vertex}

In some sense a dual to complete vertices are the isolated vertices. In this
section we study the case where the commuting graph $\Gamma (G,A)$ of $A$-orbits of $G$
has an isolated vertex $g^{A}$, that is, $g\neq 1$ and $\{g^{A}\}$ is a
connected component of $\Gamma (G,A)$. This forces that either $g\in Z(G)$
or $Z(G)=1.$ In the former case $G\subseteq g^{A}\cup \{1\}$ which implies
that $G$ is an elementary abelian group and $A$ acts transitively on $%
G\setminus \{1\}.$ That is, $\Gamma (G,A)$ has only one vertex. Throughout this section we shall assume that $Z(G)=1.$

\begin{proposition}
Suppose that $A\leq AutG$ where $Z(G)=1$ and that the
graph $\Gamma (G,A)$ has an isolated vertex $g^{A}$. Let $p$ be a prime
dividing the order of $g.$ Then $C_{G}(g)$ is a Sylow $p$-subgroup of $G$
which is an elementary abelian $CC$-subgroup of $G$, that is, for any
nonidentity element $x\in C_{G}(g)$, $C_{G}(x)=C_{G}(g).$
\end{proposition}
\begin{proof}
Observe that $C_{G}(g)\subseteq g^{A}\cup \{1\}$. Then $g^{p}$ cannot be $A$%
-conjugate to $g$ and hence $g^{p}=1$. This forces that for all $x\in g^{A}$%
, $x^{p}=1$ and so $C_{G}(g)$ is a group of exponent $p.$ Let now $P$ be a
Sylow $p$-subgroup of $G$ with $C_{G}(g)\leq P.$ If $1\neq x\in Z(P)$ then
there exists $a\in A$ such that $x^{a}=g$ and hence we get $P^{a}\leq
C_{G}(x^{a})=C_{G}(g)$ showing that $C_{G}(g)$ is a Sylow subgroup of $G.$
For any nonidentity $p$-element $y\in G$, there exist $z\in G$ and $a\in A$
such that $y^{z}=g^{a}$. Then $C_{G}(y)=C_{G}(g)^{az^{-1}}$ is a Sylow $p$%
-subgroup of $G$. It follows by Theorem C in \cite{ArChil} that a Sylow $p$%
-subgroup $C_{G}(g)$ is an elementary abelian $CC$-subgroup of $G$.
\end{proof}

Appealing to Theorem $A$ of \cite{ArChil} one can classify the groups
satisfying our hypothesis. We are not going to give any further comments on this
question.

\begin{theorem}
Suppose that $A\leq Aut(G)$ where $Z(G)=1$. Then $\Gamma (G,A)$ has
no edges and more than one vertex if and only if $G$ is either $PSL(2,5)$ or
a Frobenius group with elementary abelian kernel and complement of prime
order. Furthermore $A$ is a group of automorphisms such that for any Sylow
subgroup $P$ of $G$, the set $P\setminus \{1\}$ is an $N_{A}(P)$-orbit.
\end{theorem}

\begin{proof}
By hypothesis for any nonidentity element $ x\in G$, the vertex $x^{A}$ is an isolated vertex. By
Proposition 4.1 $x$ is of prime order and its centralizer in $G$ is a Sylow
subgroup of $G$ which is elementary abelian. Therefore $G$ is a CP-group,
that is, every element in $G$ is of prime power order, in which every Sylow
subgroup is elementary abelian. Furthermore if $P$ is a Sylow subgroup of $G$
and $1\neq g\in P$ then for any $g\neq h\in P\setminus \{1\}$ there exists $%
a\in A$ such that $g^{a}=h$, and hence $P^{a}=C_{G}(b)=P$ which shows that $%
N_{A}(P)$ acts transitively on $P\setminus \{1\}$ as claimed.

The structure of a CP-group $E$ is known (see \cite{DelWu}) and one of the
following holds:

\noindent $(i)$ $E$ is a $p$-group for some prime $p;$\\
$(ii)$ $E$ is a Frobenius group with $\left\vert \pi (E)\right\vert =2;$%
\\
$(iii)$ $E$ is a $2$-Frobenius group with $\left\vert \pi (E)\right\vert =2;$%
\\
$(iv)$ $E$ is isomorphic to one of the following groups: $PSL(2,q)$ for $q\in
\{5,7,8,9,17\}$ or $PSL(3,4)$ or $Sz(8)$ or $Sz(32)$ or $M_{10}$ or $%
O_{2}(E)\neq 1$ , $E/O_{2}(E)$ is isomorphic to one of the following groups $%
PSL(2,q)$ for $q\in \{4,8\}$ or $Sz(8)$ or $Sz(32).$ Furthermore $O_{2}(E)$
is isomorphic to a direct sum of natural modules for $E/O_{2}(E).$

As $Z(G)=1$, $G$ is not a $p$-group. Also $G$ cannot be a $2$-Frobenius
group with $\left\vert \pi (G)\right\vert =2$ and elementary abelian Sylow
subgroups because otherwise $F(G)$ must be a Sylow subgroup. If $G$ is a
Frobenius group with $\pi (G)=2$ then the kernel and the complement must be
Sylow subgroups of $G$ and as they are elementary abelian the Frobenius
complement must be cyclic of prime order. Also one can observe that all
nonsolvable groups other than $PSL(2,5)$ in the list do not satisfy the
condition that each Sylow subgroup is elementary abelian.
\end{proof}

\section{when $\Gamma (G,A)$ is triangle free}

In this section we work under the hypothesis that $\Gamma (G,A)$ has no
triangles.

\begin{lemma}
For all nonidentity elements $x\in G,$ $\left\vert x\right\vert $ divides $%
p^{2}$ for some prime $p$.
\end{lemma}

\begin{proof} $\left\vert x\right\vert $ is a power of a prime for
any $1\neq x\in G;$ because otherwise there would be an element $x\in G$
such that $|x|=pq$ for distinct primes $p$ and $q.$ Then the vertices $%
x^{A},(x^{p})^{A},(x^{q})^{A}$ form a triangle, which is contradiction. Also
if $\left\vert x\right\vert =p^{3}$ for some prime $p$ then the vertices $%
x^{A},(x^{p})^{A},(x^{p^{2}})^{A}$ form a triangle.
\end{proof}

\begin{theorem}
If $\Gamma (G,A)$ is triangle free then $G$ is a CP-group. Furthermore if $G$
is nonsolvable then either $G$ is isomorphic to one of the simple groups $%
PSL(2,q)$ for some $q\in \{5,7,8,9\}$; or $PSL(3,4)$; or has a nontrivial
normal $2$-subgroup and $G/O_{2}(G)$ is isomorphic to $PSL(2,4)$ or $%
PSL(2,8) $. In the last case $O_{2}(G)$ is isomorphic to a direct sum of
natural modules of the group $G/O_{2}(G)$.
\end{theorem}

\begin{proof}
It follows from Lemma 5.1 that $G$ is a $CP$-group. Therefore its structure is
well known (see the proof of Theorem 4.2). One can observe that $%
S=G/O_{2}(G) $ is not isomorphic to $Sz(32),M_{10}$ or $PSL(2,17),$ because
the first one of these groups contains a Sylow $31$-subgroup $R$ of order $%
31 $ with the property that $C_{S}(R)=R$ and $
N_{Aut(Sz(32))}(R)/C_{Sz(32)}(R)\cong \mathbb{Z}_{10}$ by \cite{Atl} so that there are at
least $3$ $A$-orbits of elements of order $31$ in $G$ which form a triangle,
and the last two contain elements of order $8$, which are impossible by Lemma 5.1.

We complete now the proof by showing that $\overline{G}%
=G/O_{2}(G)$ cannot be isomorphic to $Sz(8)$. Suppose the contrary. Let $T$
be a Sylow $2$-subgroup of $\overline{G}$. Then $Z(T)$ is an elementary abelian
group of order $8$ and $T\setminus Z(T)$ is the subset of $T$ consisting of
elements of order $4.$ Furthermore $N_{Aut(Sz(8))}(T)=TRS$ where $RS$ is a
Frobenius group of order $21$ with kernel $R$ and complement $S.$ $R$ acts
transitively on the set of nonidentity elements of $Z(T)$ and also on $
T/Z(T) $ and $C_{T}(S)=\left\langle x\right\rangle \cong \mathbb{Z}_{4}$. If 
$u$ and $v$ are two elements of order $4$ of $T$ and there exists $a\in A$
with $u^{a}=v$ then $(u^{2})^{a}=v^{2}$ and $u^{2},v^{2}\in Z(T)\setminus
\{1\}$ implying that $a\in N_{A}(T).$ But the number of $N_{Aut(Sz(8))}(T)$
-orbits on $T\setminus Z(T)$ is $2$, each of length $28$ and are represented
by an element $x$ of order $4$ and its inverse. Observe that $
C_{T}(x)=Z(T)\left\langle x\right\rangle .$ So we have a triangle $
\{x^{A},(x^{-1})^{A},y^{A}\} $ where $1\neq y\in Z(T)$ in $\Gamma (Sz(8),A)$
for any $A\leq Aut(Sz(8)).$ Therefore if $O_{2}(G)=1$ we have a triangle in $
\Gamma (G,A)$ for any $A\leq Aut(G),$ which is not the case.

Suppose now that $O_{2}(G)\neq 1.$ Let $u\in G$ with $uO_{2}(G)=x$ and $
1\neq z\in O_{2}(G)\cap C_{G}(u).$ Then $u^{A},(u^{-1})^{A},z^{A}$ form a
triangle for any $A\leq Aut(G)$ which is impossible.
\end{proof}

\begin{remark} We want explicitly remark that the above theorem does not say anything about the existence of $A\leq AutG$ such that $\Gamma(G,A)$ is triangle free for a given group $G$. It is an independent and interesting project to classify all such pairs $(G,A)$. 
\end{remark}

\section*{Acknowledgements}

We thank N. V. Maslova and A.\,V.~Vasil'ev for their valuable comments on our questions about the Gr\" unberg-Kegel graph.


\begin{thebibliography}{}
\bibitem{Ab} A. Abdollahi, S. Akbari, H. R. Maimani, Non-commuting graph of
a group, \textit{J. Algebra} \textbf{298} (2006) 468--492.

\bibitem{ArChil} Z.Arad, D.Chillag, On finite groups with conditions on the
centralizers of p-elements, \textit{J. Algebra} \textbf{51} (1978) 164--172.

\bibitem{Atl} J. H. Conway, R. T. Curtis, S. P. Norton, R. A. Parker, and R.
A. Wilson, Atlas of finite groups, \textit{Clarendon Pres}s, Oxford (1985).

\bibitem{Dar} M. R. Darafsheh, Groups with the same non-commuting graph, 
\textit{Discrete Appl. Math.} \textbf{157} (2009) 833--837.

\bibitem{DelWu} A. L. Delgado, Y. Wu, On locally finite groups in which
every element has prime power order, \textit{Illinois Journal of Mathematics}%
\textbf{\ 46} (2002) 885--891.

\bibitem{DoGuHePra} S. Dolfi, R. M. Guralnick, M. Herzog, and C. E. Praeger,
A new solvability criterion for finite groups, \textit{Journal of the London
Mathematical Society}\textbf{\ 85} (2) (2012) 269--281.

\bibitem{GlauGur} G. Glauberman, R. Guralnick, J. Lynd, and G. Navarro,
Centers of Sylow subgroups and automorphisms, Preprint. \textit{ArXiv.org}

\bibitem{HeLoMa} M. Herzog, P. Longobardi, M. Maj, On a commuting graph on
conjugacy classes of groups, \textit{Communications in Algebra} \textbf{37}
(10) (2009) 3369--3387.

\bibitem{Liebeck} M. W. Liebeck, C. E. Praeger, and J. Saxl, Transitive subgroups of primitive permutation groups, \textit{J. Algebra} \textbf{234} (2000) 291--361.

\bibitem{Lucido} M. S. Lucido, A. R. Moghaddamfar, Groups with complete
prime graph connected components, \textit{J. Group Theory} \textbf{7} (2004)
373--384.

\bibitem{Erf} A. Mohammadian, A. Erfanian, M. Farrokhi D. G., and B.
Wilkens, Triangle-free commuting conjugacy classes graphs, \textit{J. Group
Theory} \textbf{19} (6) (2016) 1049--1061.

\bibitem{PaMo} G. L. Morgan, C. W. Parker, The diameter of the commuting
graph of a finite group with trivial centre, \textit{J. Algebra} \textbf{393}
(2013) 41--59.

\bibitem{CP} C. W. Parker, The commuting graph of a soluble group, \textit{%
Bull. Lond. Math. Soc.} \textbf{45} (4) (2013) 839--848.

\bibitem{SW} R. M. Solomon, A. Woldar, Simple groups are characterized by
their non-commuting graph, \textit{J. Group Theory} \textbf{16} (2013)
793--824.

\bibitem{VaV} A. V. Vasiliev, E. P. Vdovin, An adjacency criterion for the
prime graph of a finite simple group, \textit{Algebra and Logic} \textbf{44}
(6) (2005) 381--406.
\end{thebibliography}
\end{document}